\documentclass[12pt]{amsart}
\usepackage{a4wide}
\usepackage{comment}
\usepackage[dvipsnames]{xcolor}
\usepackage{mathrsfs}
\usepackage{amsmath}
\usepackage{amsthm}
\usepackage{amsfonts}
\usepackage{amssymb}
\usepackage{enumitem}
\usepackage{graphicx}
\usepackage{url}
\usepackage{accents}
\usepackage{wasysym}
\usepackage{tikz}
\usepackage{palatino}
\usepackage{overpic}
\usepackage[T1]{fontenc}
\usepackage{hyperref}
\usepackage[noabbrev,capitalize,nameinlink]{cleveref}
\newtheorem{theorem}{Theorem}[section]
\newtheorem{proposition}[theorem]{Proposition}
\newtheorem{lemma}[theorem]{Lemma}

\theoremstyle{definition}

\newtheorem{question}[theorem]{Question}

\theoremstyle{remark}

\numberwithin{equation}{section}

\usepackage{mleftright}

\definecolor{darkgreen}{RGB}{0,162,0}
\definecolor{darkred}{RGB}{221,0,0}
\definecolor{darkblue}{RGB}{0,0,221}
\definecolor{gold}{RGB}{255,210,0}
\definecolor{red}{RGB}{242,43,29}
\hypersetup{colorlinks,linkcolor={blue},citecolor={darkgreen},urlcolor={darkgreen}}

\begin{document}

\title{Non-orientable Nurikabe}

\author{Joseph Breen}
\address{University of Alabama, Tuscaloosa, AL 35401}
\email{jjbreen@ua.edu} \urladdr{https://sites.google.com/view/joseph-breen}

\author{Emma Copeland}
\address{University of Iowa, Iowa City, IA 52240}
\email{emma-r-copeland@uiowa.edu} 

\thanks{JB was partially supported by NSF Grant DMS-2038103 and an AMS-Simons Travel Grant. EC was partially supported by a University of Iowa URA award.}

\begin{abstract}
    We study Nurikabe puzzles on non-orientable surfaces. Specifically, we propose two versions of non-orientable Nurikabe and investigate their combinatorics on Möbius strips, Klein bottles, and projective planes of size $1\times n$. Our results establish new connections among the OEIS sequences A101946, A213387, A123203, and A001045 (the Jacobsthal sequence). 
\end{abstract}

\maketitle

\tableofcontents

\section{Introduction}

Nurikabe is a pencil puzzle, published and popularized by the Japanese company Nikoli alongside the likes of Sudoku, Slitherlink, and Shikaku \cite{boswell2022countingislands}. A Nurikabe puzzle begins with an $m\times n$ rectangular grid with a subset of squares populated by positive integers. A solution is a coloring of some of the empty squares black, so that the following rules are satisfied: 
\begin{itemize}
    \item[(N1)] The set of of black squares is orthogonally connected; here, two squares are \textit{orthogonally connected} (or \textit{adjacent}) if they share an edge. 
    \item[(N2)] The subset of black squares contains no $2\times 2$ sub-grids. 
    \item[(N3)] The complement of the black squares consists of orthogonally connected components, each containing a single integer denoting the number of squares in the connected component. 
\end{itemize}
The subset of black squares is referred to as the \textit{water}, a $2\times 2$ sub-grid of black squares is called a \textit{whirlpool} (see \cref{fig:wp}), and a \textit{river} is an orthogonally connected water region without whirlpools. An orthogonally connected component of the complement of the water is an \textit{island} consisting of \textit{land} squares. Thus, a solution to a Nurikabe puzzle is a single river of water determining orthogonally connected islands of prescribed sizes; see \cref{fig:nurikabe1} for an example.

\begin{figure}[ht]
	\centering
    \vskip-0.45cm
	\begin{overpic}[scale=.38]{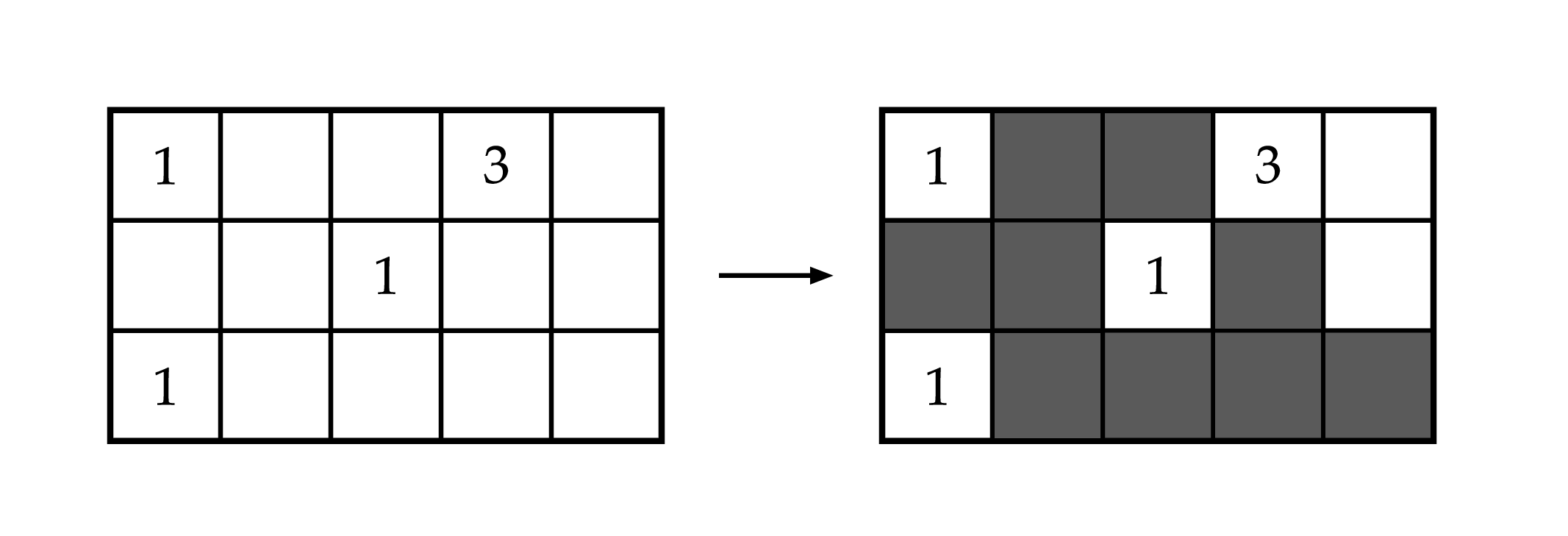}
        
	\end{overpic}
    \vskip-0.4cm
	\caption{A Nurikabe puzzle (left) and its solution (right).}
	\label{fig:nurikabe1}
\end{figure}

Academic studies of Nurikabe have focused on establishing algorithmic solutions \cite{amos2019ant,robert2021zkp} and studying computational complexity --- for instance, Nurikabe is NP-complete \cite{mcphail2003complexity,holzer2011computational}. Combinatorial explorations include counting problems for islands of size $1$ \cite{boswell2022countingislands}, and, most relevant result for us, an enumeration result of Goertz and Williams \cite{goertz2024quaternarygraycodeused} arising in the study of certain Ziggu puzzles. Specifically, the latter authors demonstrate a bijection between Ziggu states and the set $\mathcal{N}_n$ of all $2\times n$ \textit{Nurikabe rectangles} --- i.e.\ $2\times n$ rectangular black-white grids satisfying (N1,N2) --- counting the latter and identifying the result with \cite[A101946]{oeis}:
\begin{equation}\label{eq:N_n}
N_n := |\mathcal{N}_n| = 6\cdot 2^n - 3n - 5.  
\end{equation}
While Sudoku has attracted a significant amount of mathematical attention (see \cite{rosenhouse2011taking} for a survey), other puzzles like Nurikabe are comparatively less explored. 

\begin{figure}[ht]
	\centering
    \vskip-0.45cm
	\begin{overpic}[scale=.38]{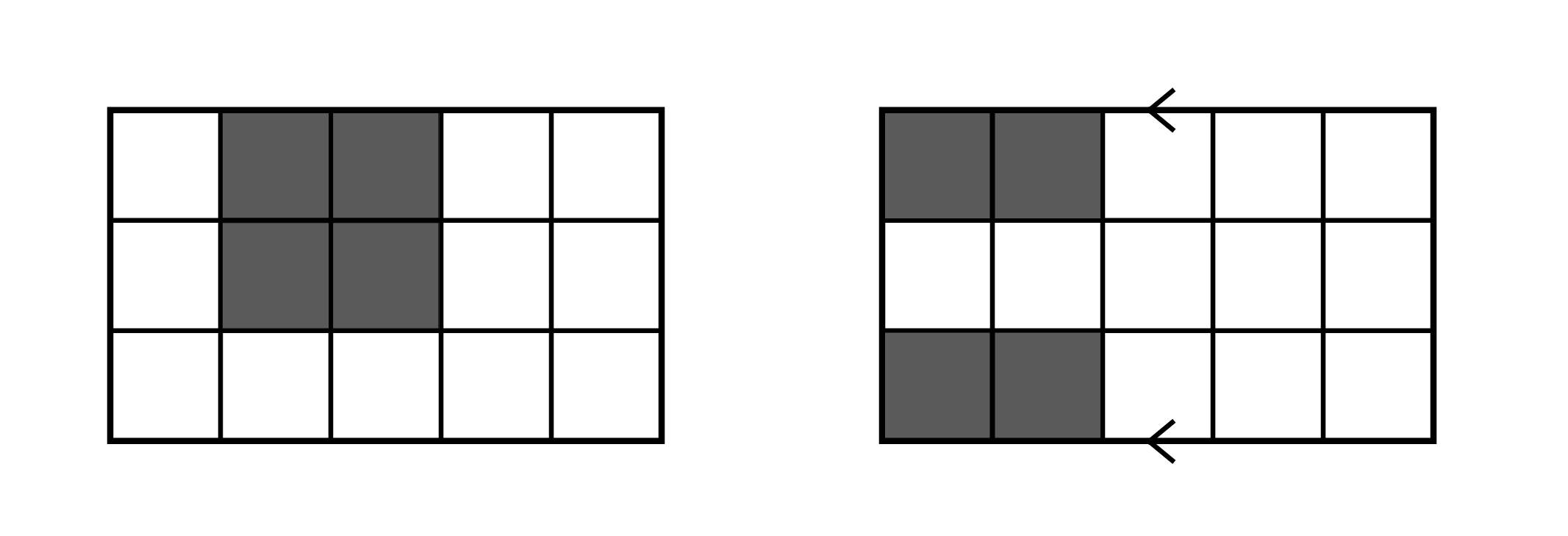}
        
	\end{overpic}
    \vskip-0.4cm
	\caption{A whirlpool on the rectangle (left) and a whirlpool on the annulus (right).}
	\label{fig:wp}
\end{figure}

\subsection{Non-orientable Nurikabe}

By identifying sides, we may view a Nurikabe grid as representing a fundamental domain for an annulus, a torus, a Möbius strip, a projective plane, or a Klein bottle. Toroidal Nurikabe has appeared recreationally in \cite{lance2014toroidal}, but to the best of the authors' knowledge there are no other recorded interactions of Nurikabe with non-trivial topology. 

In the orientable case (the annulus and torus) the Nurikabe rules (N1, N2) extend naturally, thanks to the observation that every \textit{interior vertex} (a corner of a square in the topological interior of the surface, as opposed to a \textit{boundary} vertex) on an annulus or torus obtained from a square-tiled rectangular fundamental domain has \textit{square-degree} $4$, i.e.\ is incident to four distinct squares, not counting multiplicity. However, the meaning of \textit{whirlpool} in the non-orientable case admits multiple interpretations, as interior vertices can have square-degree other than $4$; see, for example, \cref{fig:nurikabe2}. In particular, should a whirlpool correspond to an orthogonally connected $2\times 2$ subgrid of distinct water squares, or simply an interior vertex of any square-degree surrounded by water? 

\begin{figure}[ht]
	\centering
    \vskip-.6cm
	\begin{overpic}[scale=.44]{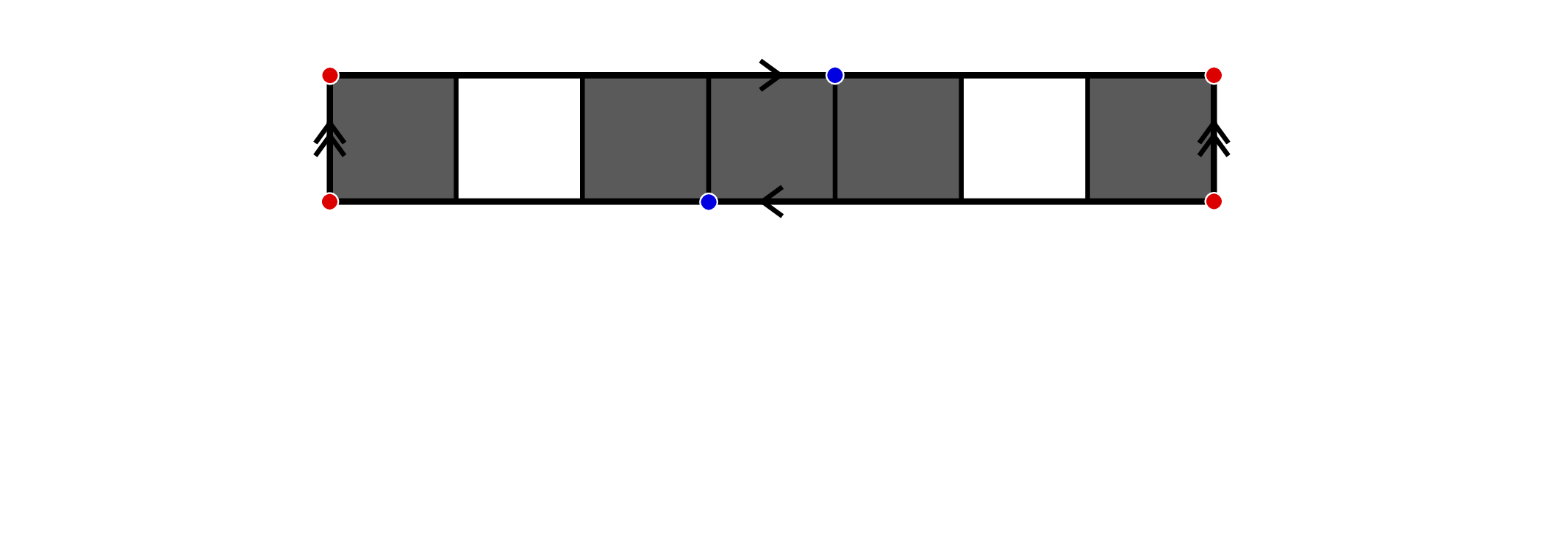}
        
	\end{overpic}
    \vskip-3cm
	\caption{A $1\times 7$ Klein bottle. The red vertex has square-degree $2$, the blue vertex has square-degree $3$, and each is surrounded by water. The grid satisfies (N2$\Box$), but does not satisfy (N2$\ocircle$).}
	\label{fig:nurikabe2}
\end{figure}

We propose two extensions of (N2), as follows: 
\begin{itemize}
    \item[(N2$\Box$)] The water contains no $2\times 2$ sub-grids of four distinct squares. Equivalently, there are no interior vertices of square-degree $4$ entirely surrounded by water. 

    \item[(N2$\ocircle$)] There are no interior vertices of any square-degree surrounded by water.
\end{itemize}
We refer to these extensions as the \emph{square} and \emph{loop} whirlpool rules, respectively. 

\begin{figure}[ht]
	\centering
	\begin{overpic}[scale=.38]{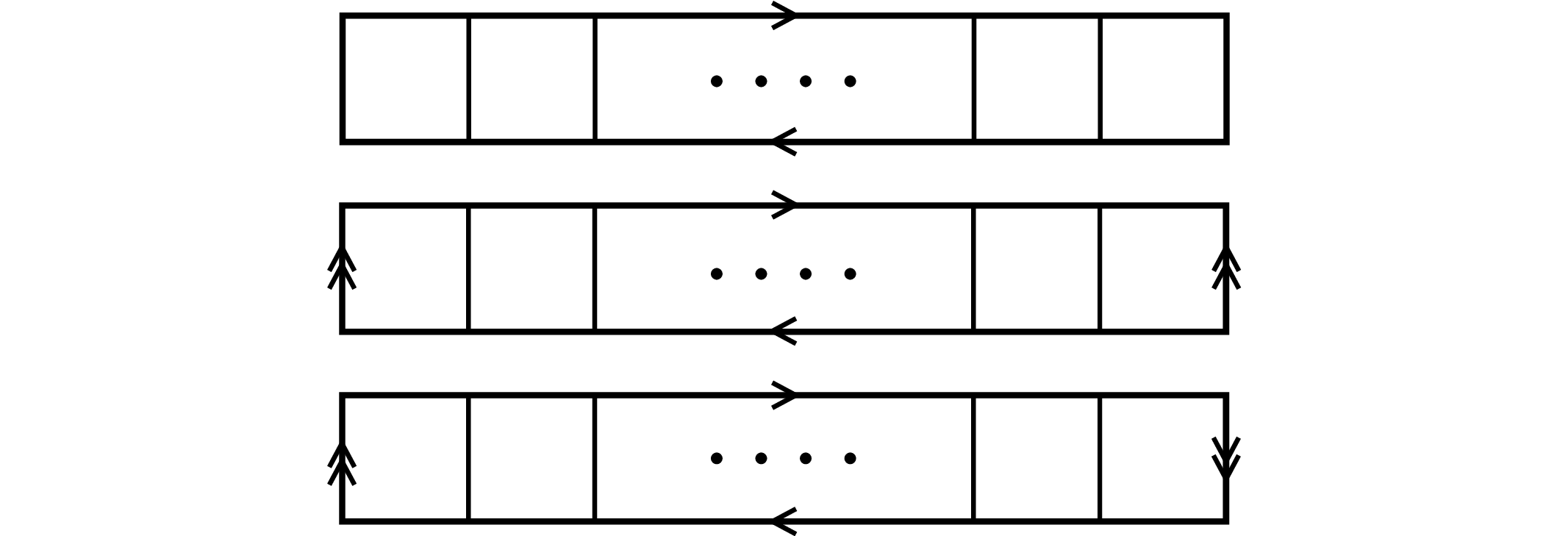}
        
	\end{overpic}
	\caption{Gluing conventions for $1\times n$ Möbius strips, Klein bottles, and projective planes, respectively.}
	\label{fig:nurikabe3}
\end{figure}

\subsection{Main results}


Recall that $\mathcal{N}_k$ denotes the set of $2\times k$ Nurikabe rectangles and $N_k := |\mathcal{N}_k|$. Let $\ast \in \{\Box,\ocircle\}$. Let $\accentset{\ast}{\mathcal{M}}_n$, $\accentset{\ast}{\mathcal{K}}_n$ and $\accentset{\ast}{\mathcal{P}}_n$ denote the sets of black and white fundamental domains of size $1\times n$ which satisfy (N1, N2$\ast$) upon gluing according to the respective conventions in \cref{fig:nurikabe3}. We informally refer to these as the sets of square/loop Nurikabe Möbius strips, Klein bottles, and projective planes, respectively. Likewise, we denote $\accentset{\ast}{M}_n := |\accentset{\ast}{\mathcal{M}}_n|$, $\accentset{\ast}{K}_n := |\accentset{\ast}{\mathcal{K}}_n|$, and  $\accentset{\ast}{P}_n := |\accentset{\ast}{\mathcal{P}}_n|$. 

\begin{theorem}\label{thm:main}
Let $N_{k} = 6\cdot 2^k - 3k - 5$ be the count of $2\times k$ Nurikabe rectangles from \eqref{eq:N_n}, and let $J_k = \frac{2^k - (-1)^k}{3}$ denote the Jacobsthal sequence. Let $n\geq 1$. Using the square whirlpool rule (N2$\Box$), we have
\begin{equation}\label{eq:main_square}
\accentset{\Box}{M}_n = \accentset{\Box}{K}_n = \accentset{\Box}{P}_n = \begin{cases}
 N_k   & \quad \text{if }n=2k \\
  N_k + 3\cdot 2^k - 2  & \quad \text{if }n=2k+1
\end{cases}.  
\end{equation}
Using the loop whirlpool rule (N2$\ocircle$), we have
\begin{align}\label{eq:main_loop}
\accentset{\ocircle}{M}_n &= \begin{cases}
  N_k - 2^k + 1 & \quad \text{if }n=2k \\
  N_k + 2^{k+1} - 1  & \quad \text{if }n=2k+1
\end{cases},\\
\accentset{\ocircle}{K}_n = \accentset{\ocircle}{P}_n &= \begin{cases}
\, 1 & \quad \text{if }n=1 \\
\, 3 & \quad \text{if }n=2 \\
\, 6 & \quad \text{if }n=3 \\
\, 7 & \quad \text{if }n=4 \\
  N_k  + 2J_{k-1} - 2^{k+1} +2 & \quad \text{if }n=2k\geq 6 \\
   N_k - 2J_k + 2^k & \quad \text{if }n=2k+1\geq 5
\end{cases}.\label{eq:main_loop2}
\end{align}
\end{theorem}

The enumeration $\accentset{\ocircle}{M}_{2k} = 5 \cdot 2^{k} - 3k - 4$ of even-length loop Nurikabe Möbius strips in \eqref{eq:main_loop} is identified with \cite[A213387]{oeis}, and the odd-length enumeration $\accentset{\ocircle}{M}_{2k+1} = 2^{k+3} - 3(k+2)$ is identified with \cite[A123203]{oeis}. Thus, \cref{thm:main} establishes novel links between these sequences, the sequence \cite[A101946]{oeis}, and the ubiquitous Jacobsthal sequence \cite[A001045]{oeis}. The loop Klein count in \eqref{eq:main_loop2}, either in entirety or restricting to even and odd indices, and the odd square enumeration $\accentset{\Box}{M}_{2k+1} = 9\cdot 2^k - 3k - 7$ in \eqref{eq:main_square} do not appear to have database matches. For concrete reference, the first $12$ terms of the loop Klein sequence are: 
\[
\left\{\accentset{\ocircle}{K}_n\right\}_{n=1}^{\infty} = \{1,3,6,7,15,22,36,55,85,120,182,257,\dots\}.
\]

\subsection{Further questions}

Naturally, one can ask for Nurikabe enumerations on surfaces of larger size. Even restricting to rectangles, the authors are unaware of enumerations in the literature beyond \eqref{eq:N_n} of \cite{goertz2024quaternarygraycodeused}. 

\begin{question}
What is the number $N_{3\times n}$ of Nurikabe rectangles of size $3\times n$? What about of size $n\times n$? What about Nurikabe Möbius strips of size $3\times n$ in terms of $N_{3\times n}$?
\end{question}

A different direction is to expand the class of surfaces on which Nurikabe is played. The annuli and tori constructed from tiled rectangular fundamental domains are examples of \textit{square-tiled surfaces}, which are further generalized by \textit{translation surfaces}. Briefly, a translation surface is obtained from a fundamental set of polygons in the plane, where all edge identifications are orientation-preserving along parallel edges of same length. A square-tiled surface is a translation surface where each polygon is a square; for an example, see \cref{fig:squaretile}. Translation surfaces have deep connections to Teichmüller theory, dynamics, and combinatorics, for instance playing a central role in the celebrated work of Eskin, Mirzakhani, and Mohammadi \cite{eskin2015isolation}. For more information, we point the reader to the broad introductory book of Forni and Matheus \cite{ForniMatheus2018} and the survey article of Wright \cite{wright2016rational}.

\begin{figure}[ht]
	\centering

	\begin{overpic}[scale=.36]{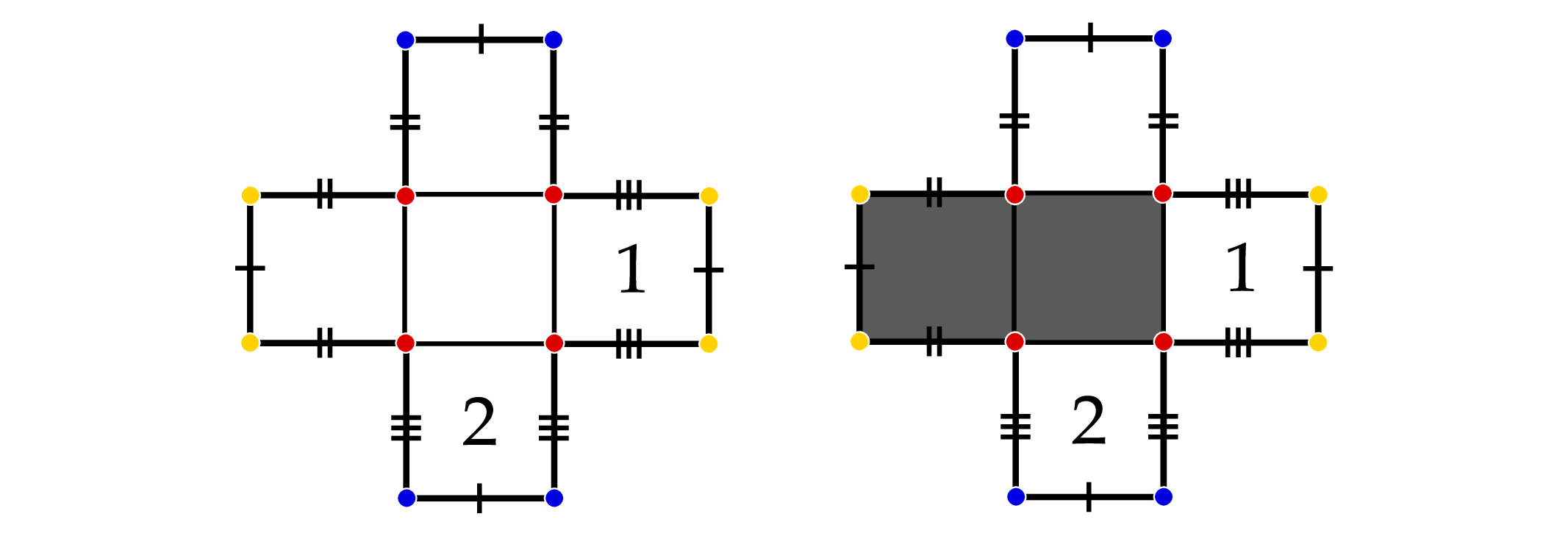}
        
	\end{overpic}
    
	\caption{A Nurikabe puzzle on a genus $2$ square-tiled surface (left) and its unique solution (right).}
	\label{fig:squaretile}
\end{figure}

The loop whirlpool Nurikabe rule set (which is evidently more interesting) extends naturally to square-tiled surfaces, and even to translation surfaces in general under reasonable reinterpretation. Beyond the basic enumeration problems as in the present paper, there are many questions which are natural to ask. To provide an example, note that in a $2\times n$ grid there are $2^{2n} = 4^n$ possible black/white colorings. As $n\to \infty$, \eqref{eq:N_n} implies that the proportion of Nurikabe grids among all colorings tends to $0$: $\frac{N_n}{4^n} \to 0$. More generally, for a square-tiled (or translation) surface $\Sigma$, let $N(\Sigma)$ denote the number of valid loop Nurikabe colorings based on $\Sigma$ and let $|\Sigma|$ denote the number of squares (or polygons) in $\Sigma$. 

\begin{question}
Does there exist a sequence of square-tiled or translation surfaces $\Sigma_{n}$ with $|\Sigma_n| \to \infty$ such that $\liminf \frac{N(\Sigma_n)}{2^{|\Sigma_n|}} > 0$?
\end{question}

\noindent Here the immediate obstruction appears to be orthogonal connectedness of the water. For instance, the staircase square-tiled surface in \cref{fig:staircase} has only one vertex, so it is impossible to have a loop whirlpool unless all squares are water; orthogonal connectedness however does not hold for a general set of water squares.  

\begin{figure}[ht]
	\centering

	\begin{overpic}[scale=.4]{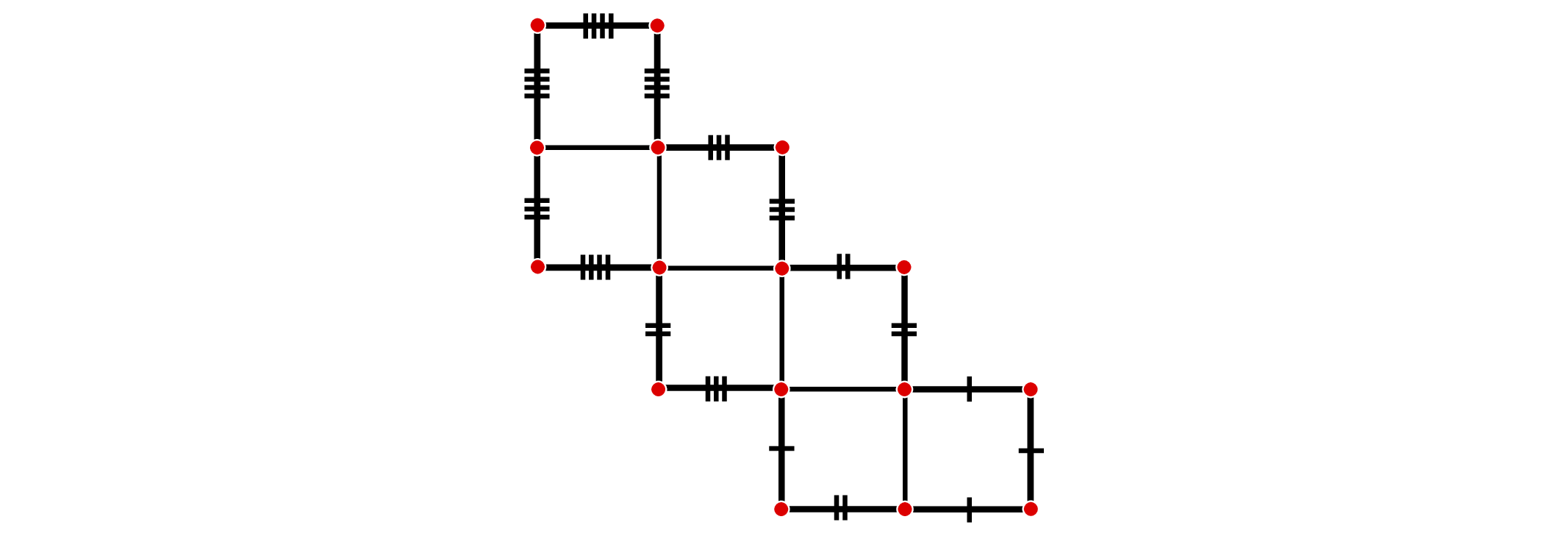}
        
	\end{overpic}
    
	\caption{A staircase square-tiled surface with one vertex (in red).}
	\label{fig:staircase}
\end{figure}

\subsection{Outline}

In \cref{sec:redcontr} we study orthogonal connectedness of water as in (N1), wh\-ich is indifferent to the choice of whirlpool rule. In the process of doing so we introduce certain cutting operations to reduce to $2\times k$ rectangles. These reductions require various refinements of $2\times k$ Nurikabe rectangle enumerations, which we establish in \cref{sec:refined}. We use the results from these sections to complete the proof of \cref{thm:main} in \cref{sec:square_whirl} and \cref{sec:loop_whirl}, where we consider (N2$\Box$) and (N2$\ocircle$) respectively.

\subsection{Conventions}\label{subsec:conventions}

We henceforth say \textit{(N1)-connected}, or simply \textit{connected} when clear, in place of \textit{orthogonally connected}. We use \textit{(square/loop) Nurikabe} as an adjective to indicate a surface with a coloring satisfying (N1) and (N2$\ast$) with $\ast\in \{\Box,\ocircle\}$ respectively, at times also speaking of \textit{Nurikabe validity} to mean the same. For example, a "loop Nurikabe Klein bottle" is an element $K\in \accentset{\ocircle}{\mathcal{K}}_n$, or we may conclude that a Möbius strip $M$ with a coloring of squares "is loop Nurikabe" if $M\in \accentset{\ocircle}{\mathcal{M}}_n$ or is "not loop Nurikabe valid" if $M\notin \accentset{\ocircle}{\mathcal{M}}_n$. In contrast, when referring the underlying square-tiled surface described by the canonical fundamental domains in \cref{fig:nurikabe3}, indifferent to a coloring or lack thereof, we will say, for example, "let $M$ be a $1\times n$ Möbius strip." 

Given a $2\times k$ fundamental domain, we index the columns from $1$ to $k$, ordered from left to right. On a $1\times n$ fundamental domain, we refer to the square in \textit{position (index)} $j$, or the $j$ \textit{square}, ordered from left to right. 

Finally, given a surface with square-tiled fundamental domain, we refer to its \textit{adjacency graph (of squares)} as the graph whose vertices correspond to distinct squares and whose edges indicate the presence of an adjacency between the corresponding squares, ignoring multiplicity and self-adjacencies. For example, see \cref{fig:adjgraph}. 

\begin{figure}[ht]
	\centering
	\begin{overpic}[scale=.44]{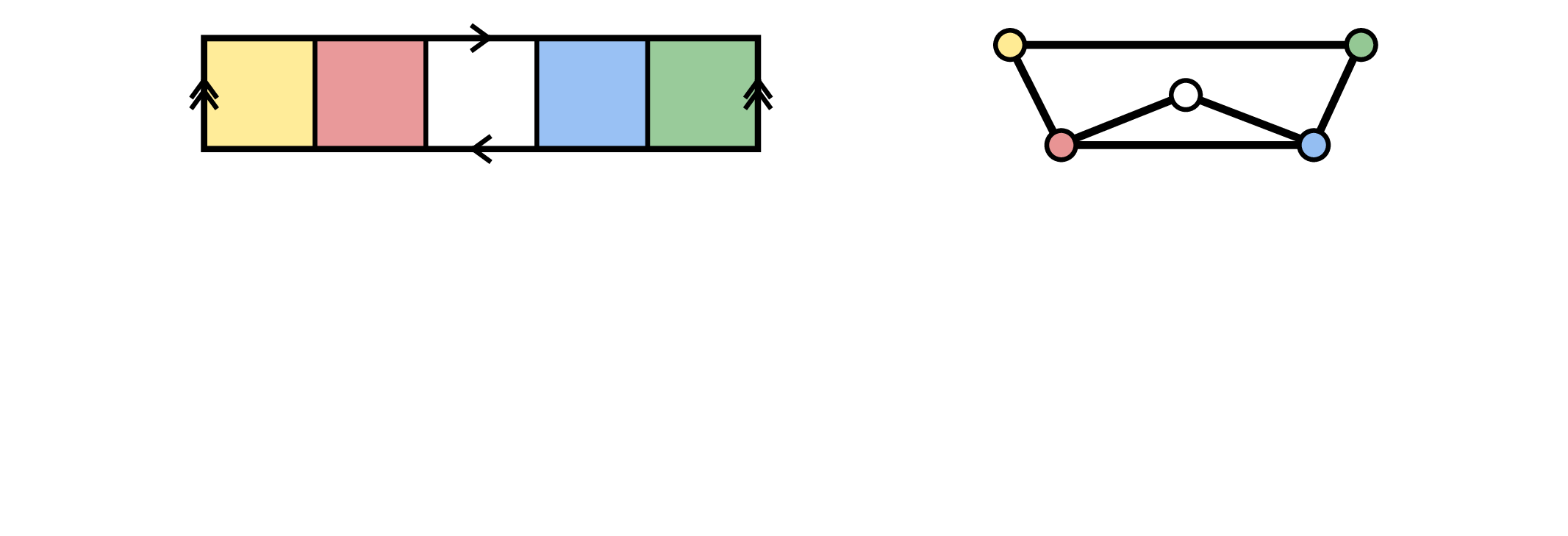}
        
	\end{overpic}
    \vskip-3.5cm
	\caption{A $1\times 5$ Klein bottle and its adjacency graph.}
	\label{fig:adjgraph}
\end{figure}
    
\subsection{Acknowledgments} The authors would like to thank Michael Dougherty for a helpful discussion, the suggestion to consider translation surfaces, and commenting on an early draft. We additionally thank an anonymous referee for multiple suggestions which improved the clarity of the article. 
\section{Reduction and contraction}\label{sec:redcontr}

We begin by studying (N1)-connectedness, which is insensitive to the choice of whirl\-pool rule (N2$\Box$) or (N2$\ocircle$). \cref{subsec:AandR} establishes \cref{lemma:orientable_rect} on rectangles and annuli. In \cref{subsec:mob_con} we consider even and odd Möbius strips separately using two cutting operations (\cref{lemma:connected_red} and \cref{lemma:connected_contr}). Finally in \cref{subsec:KP_con} we observe that (N1)-connectedness is equivalent on the three non-orientable surfaces; see \cref{lemma:KPM_connected}. Broadly speaking, this section establishes (N1)-connectedness as not troublesome in our setting and allows us to dedicate the rest of the article to whirlpools. 

\subsection{Rectangles and annuli}\label{subsec:AandR}

We first observe that a $2\times k$ domain is Nurikabe valid whether it is viewed as a rectangle or an annulus, independent of which whirlpool rule is used. In line with the notation introduced before \cref{thm:main}, let $\accentset{\ast}{\mathcal{A}}_k$ denote the set of $2\times k$ black and white fundamental domains satisfying (N1, N2$\ast$) after identifying the horizontal edges.

\begin{lemma}\label{lemma:orientable_rect}
On rectangles and annuli the Nurikabe rules (N2$\Box$) and (N2$\ocircle$) are equivalent. In particular, $\mathcal{A}_k := \accentset{\Box}{\mathcal{A}}_k = \accentset{\ocircle}{\mathcal{A}}_k$. Moreover, there is a natural bijection $\mathcal{N}_k \cong \mathcal{A}_k$.
\end{lemma}

\begin{proof}
That (N2$\Box$) and (N2$\ocircle$) are equivalent follows from the fact that, for an orientable surface obtained from a square-tiled rectangular fundamental domain, every interior vertex has square-degree $4$. In particular, an interior vertex is surrounded by water if and only if it is the center of a $2\times 2$ sub-grid of four distinct water squares. 

\begin{figure}[ht]
	\centering
    \vskip-0.2cm
	\begin{overpic}[scale=.4]{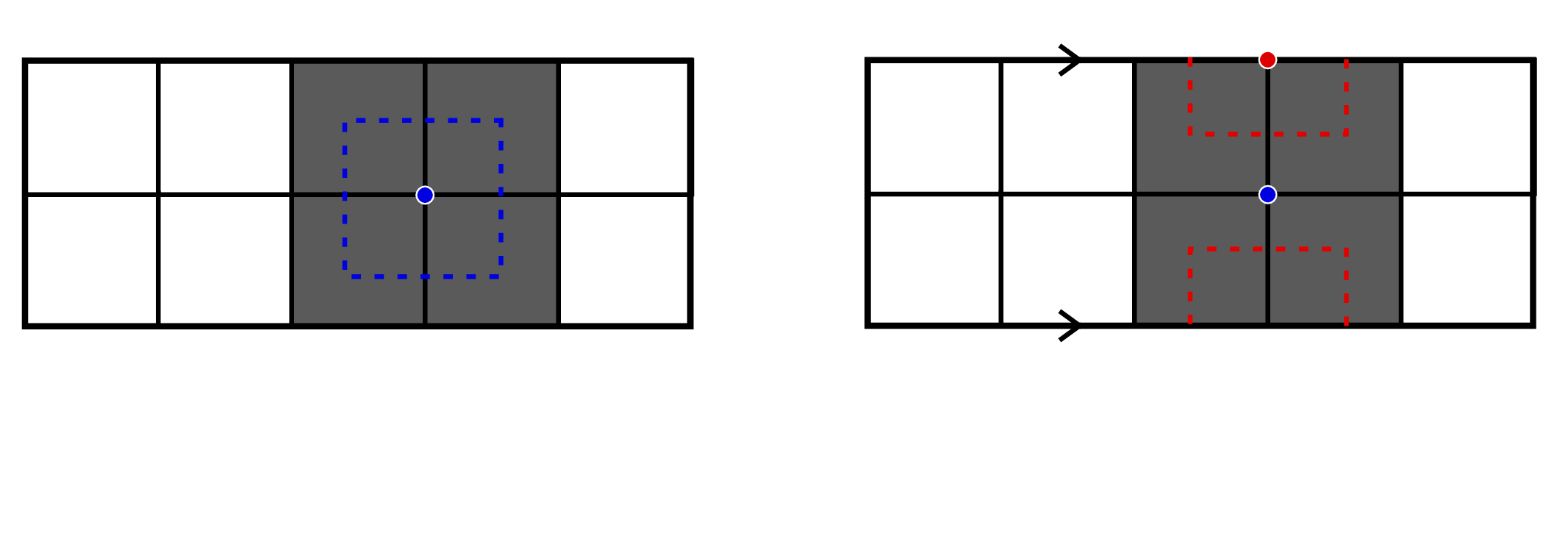}
        
	\end{overpic}
    \vskip-1.1cm
	\caption{On the left, an interior vertex (blue) of a $2\times 5$ rectangle surrounded by a $2\times 2$ whirlpool. Gluing the horizontal edges as on the right produces a new interior vertex (red) surrounded by water, but involving the same set of water squares.}
	\label{fig:annulus}
\end{figure}

Next consider a $2\times k$ rectangular grid populated by some water squares. We claim that the water is connected on the rectangle if and only if it is connected on the annulus obtained by identifying the horizontal edges. Indeed, a given square on the rectangle is adjacent to its neighbor in the same column; identifying the horizontals to produce the annulus has no effect to the adjacency graph. Likewise, identifying the horizontals creates no new $2\times 2$ sub-grids. While it does create additional interior vertices, the associated $2\times 2$ sub-grids coincide set-wise with sub-grids in the rectangle; see \cref{fig:annulus}. Thus, existence of whirlpools is equivalent on the rectangle or annulus. 
\end{proof}

\subsection{Möbius strips}\label{subsec:mob_con}

Next we study (N1)-connectedness on Möbius strips. We consider even and odd sizes separately, defining two respective cutting operations. 

\subsubsection{Rectangular reduction}\label{subsubsec:red}

Consider a $1\times 2k$ Möbius strip $M$. (Recall by \cref{subsec:conventions} that this means $M$ has a chosen fundamental domain as in \cref{fig:nurikabe3}.) By cutting the fundamental domain along the central vertical edge shared by the $k$ and $k+1$ squares, we may re-identify edges to obtain a $2\times k$ annulus as on the right side of \cref{fig:rectred}. By further cutting along the horizontal identification of the annulus, we obtain a $2\times k$ rectangle called the \textit{rectangular reduction} of $M$, denoted $\mathrm{red}(M)$.

\begin{figure}[ht]
	\centering
	\begin{overpic}[scale=.44]{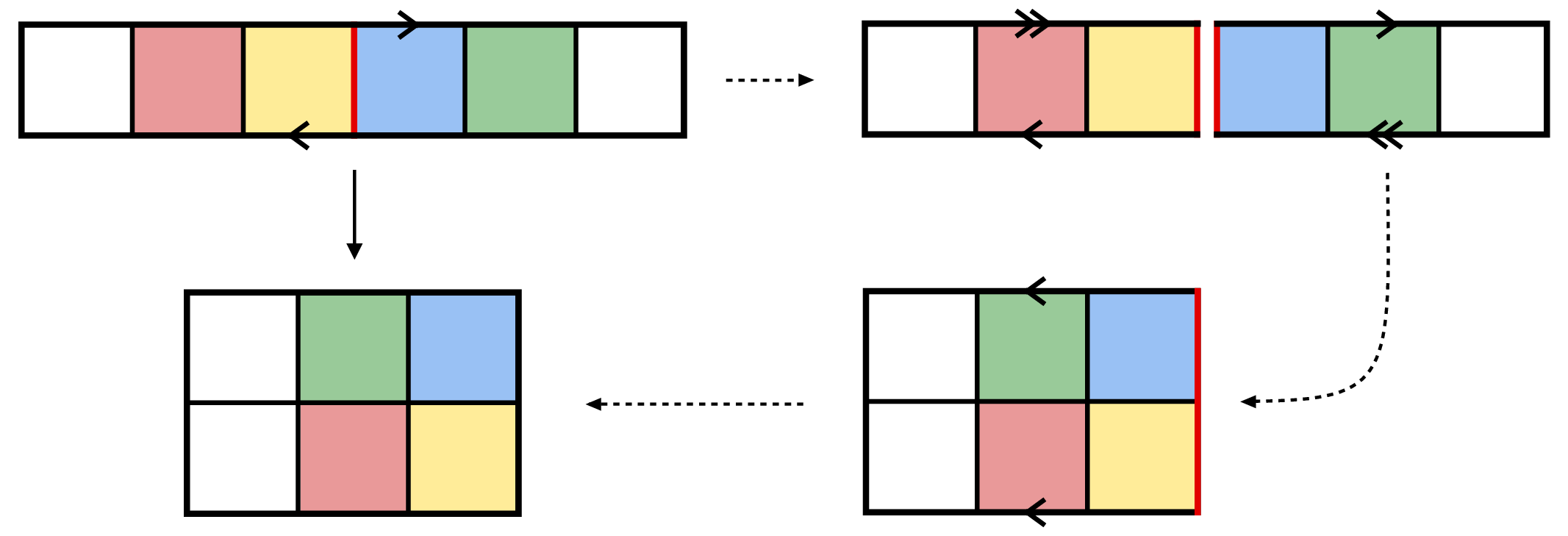}
        \put(24,20){\small $\mathrm{red}(\cdot)$}
        \put(47,30){\small cut}
        \put(88,8){\small re-identify}
	\end{overpic}
	\caption{Rectangular reduction applied to a $1\times 6$ Möbius strip, passing through an annulus by first cutting along the red edge. Squares are colored to clarify the adjacencies at each step.} 
	\label{fig:rectred}
\end{figure}

\begin{lemma}\label{lemma:connected_red}
Let $M$ be a $1\times 2k$ Möbius strip. A subset of squares is connected on $M$ if and only if its image in the rectangular reduction $\mathrm{red}(M)$ is connected.
\end{lemma}

\begin{proof}
A set of squares on a $1\times 2k$ Möbius strip is connected if and only if it is connected on the annulus obtained by cutting along the central edge, as the $k$ and $k+1$ squares are also vertically adjacent; this cut therefore has no effect on the adjacency graph. By \cref{lemma:orientable_rect}, this is equivalent to connectedness on the induced $2\times k$ rectangle.
\end{proof}

\subsubsection{Contraction}\label{subsubsec:contr}

Consider a $1\times (2k+1)$ Möbius strip $M$. The \textit{contraction} of $M$, denoted $\mathrm{contr}(M)$, is the $1\times 2k$ Möbius strip obtained by excising the central square (in position $k+1$) from the fundamental domain and then gluing along the excision; see \cref{fig:contract}.

\begin{figure}[ht]
	\centering
	\begin{overpic}[scale=.44]{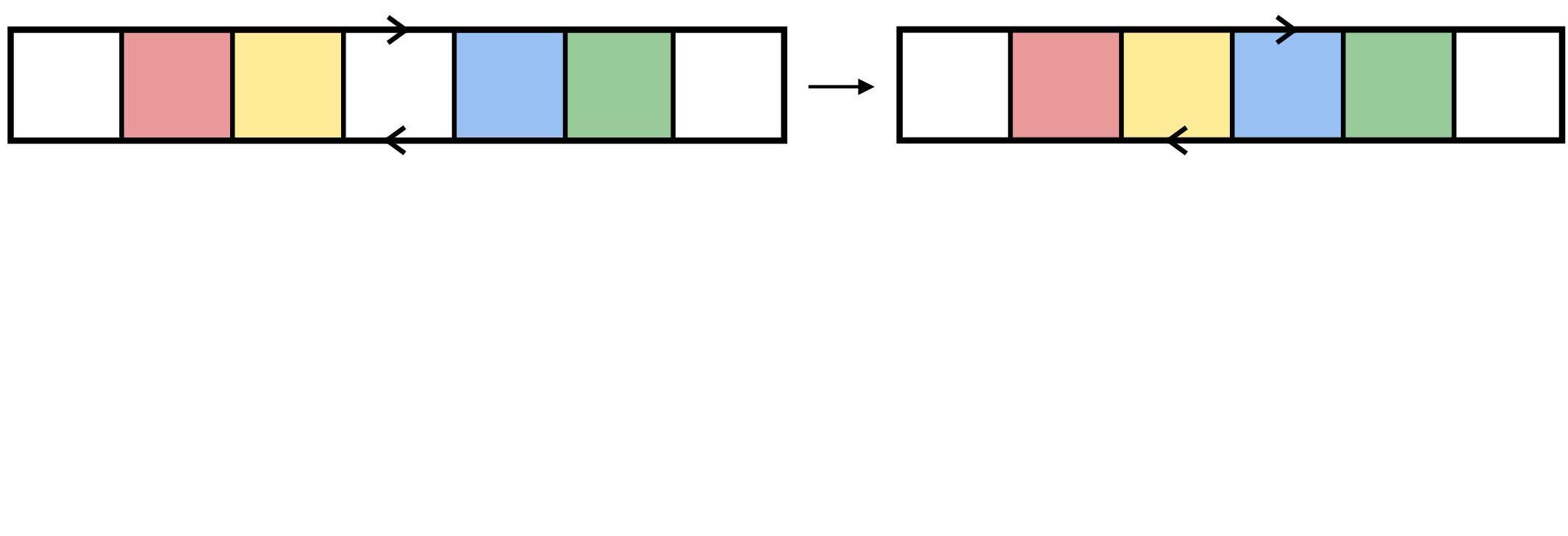}
      
	\end{overpic}
    \vskip-3.5cm
	\caption{Contraction of a $1\times 7$ Möbius strip to a $1\times 6$ Möbius strip.} 
	\label{fig:contract}
\end{figure}

\begin{lemma}\label{lemma:connected_contr}
Let $M$ be a $1\times (2k+1)$ Möbius strip. A subset of squares not containing the central $k+1$ square is connected on $M$ if and only if its image in $\mathrm{contr}(M)$ is connected. 
\end{lemma}

\begin{proof}
In $M$, the central $k+1$ square is self-adjacent in the vertical direction, and the square in position $1\leq j \leq k$ is vertically-adjacent to the square in position $2k+1 - (j-1)$. Contraction of the central square produces an even-length Möbius strip with vertical adjacencies of the square in positions $j'$ and $2k - (j'-1)$ for $1\leq j'\leq k$. Since contraction reduces by $1$ the position index of each square in $M$ of position $\geq k+1$, all vertical adjacencies are preserved. The only horizontal adjacency affected by contraction is that of the $k,k+2$ squares in $M$, which are already adjacent in the vertical direction. Thus, contraction affects the adjacency graph by deleting the vertex corresponding to the central $k+1$ square and its two incident edges.     
\end{proof}

For $\ast \in \{\Box,\ocircle\}$, decompose $\accentset{\ast}{\mathcal{M}}_{2k+1} = \accentset{\ast}{\mathcal{M}}_{2k+1,0} \cup \accentset{\ast}{\mathcal{M}}_{2k+1,1}$, where the first (resp. second) subset consists of $\ast$-Nurikabe Möbius strips whose central square is land (resp. water). 

\begin{lemma}\label{lemma:odd_M_both}
For $\ast \in \{\Box,\ocircle\}$, there is a bijection $\mathrm{red} \circ \mathrm{contr}: \accentset{\ast}{\mathcal{M}}_{2k+1,0} \to \mathcal{N}_k.$
\end{lemma}

\begin{proof}
To see that the image of $\mathrm{red} \circ \mathrm{contr}$ indeed lands in $\mathcal{N}_k$, \cref{lemma:connected_contr} reduces to checking that no $\ast$-whirlpools are created under the composition of contraction and reduction. This is clear from the fact that no interior vertices are created. The map is injective by construction, so we only need to verify surjectivity. 

Let $R\in \mathcal{N}_k$. Let $M$ denote the Möbius strip obtained by reversing reduction on $R$ and then reversing contraction with a central land square. As the water is connected in $R$, \cref{lemma:connected_contr} implies that the water is connected on $M$. \cref{lemma:orientable_rect} implies that no whirlpools (of either kind) exist on the annulus obtained from $R$. Therefore, the only way a whirlpool could appear in $M$ is as a loop whirlpool surrounding one of the interior vertices adjacent to the central square; all other interior vertices share the behavior of interior vertices on the annulus. As the central square is land, no  whirlpools are created and thus $M\in \accentset{\ast}{\mathcal{M}}_{2k+1,0}$. Therefore, we have a bijection $\accentset{\ast}{\mathcal{M}}_{2k+1,0} \to \mathcal{N}_k$.     
\end{proof}

\subsection{Klein bottles and projective planes}\label{subsec:KP_con}

Finally, we show that connectedness of water is insensitive to the difference between the three non-orientable surfaces of size $1\times n$. 

\begin{lemma}\label{lemma:KPM_connected}
Consider a $1\times n$ fundamental domain representing either a Möbius strip, Klein bottle, or a projective plane. A subset of squares is either connected on all three surfaces or disconnected on all three surfaces.   
\end{lemma}

\begin{proof}
Consider a $1\times n$ Möbius strip. The further identification of the vertical sides to produce a Klein bottle or projective plane does not change the adjacency graph; the vertical edge, once identified, is shared by squares $1$ and $n$, but these are already adjacent in the vertical direction on the Möbius strip. Therefore, connectedness of a subset of squares is independent of whether the fundamental domain is glued to a Möbius strip, Klein bottle, or projective plane.     
\end{proof}

\section{Refined rectangle counts}\label{sec:refined}

The operations in the previous section allow us to reduce to studying $2\times k$ rectangles. For this we need refinements of the enumeration $N_k$ from \eqref{eq:N_n}.

\subsection{Rightmost column restriction}\label{subsec:refined_count}

For $k\geq 1$ and $i\in \{0,1,2\}$, let $\mathcal{N}_{k,i}$ denote the set of $2\times k$ Nurikabe rectangles such that precisely $i$ squares in column $k$ are water. Denote $N_{k,i} = |\mathcal{N}_{k,i}|$.

\begin{proposition}\label{prop:a_n}
For $k\geq 1$, we have $N_{k,1} = 2^{k+1} - 2$ and $N_{k,2} = N_{k-1,1} + 1 = 2^k - 1$.    
\end{proposition}

To prove \cref{prop:a_n}, we appeal to a recursive description of the sequence. 

\begin{lemma}\label{lemma:a_n}
For $k\geq 4$, the sequence $a_k := N_{k,1}$ satisfies
\[
\begin{cases}
a_k = 2a_{k-1} + a_{k-2} - 2a_{k-3} \\
a_3 = 14 \\
a_2 = 6 \\
a_1 = 2
\end{cases}.
\]
\end{lemma}

With \cref{lemma:a_n}, the proof of \cref{prop:a_n} is a standard diagonalization argument. 

\begin{proof}[Proof of \cref{prop:a_n}.]
Observe that 
\[
\begin{pmatrix}
    a_k \\
    a_{k-1} \\
    a_{k-2}
\end{pmatrix} = \begin{pmatrix}
    2 & 1 & -2 \\
    1 & 0 & 0 \\
    0 & 1 & 0 
\end{pmatrix}\begin{pmatrix}
    a_{k-1}\\
    a_{k-2} \\
    a_{k-3}
\end{pmatrix}, \quad \text{hence} \quad \begin{pmatrix}
    a_k \\
    a_{k-1} \\
    a_{k-2}
\end{pmatrix} = \begin{pmatrix}
    2 & 1 & -2 \\
    1 & 0 & 0 \\
    0 & 1 & 0 
\end{pmatrix}^{k-3}\begin{pmatrix}
    14\\
    6 \\
    2
\end{pmatrix}.
\]
Diagonalizing the matrix gives 
\begin{align*}
\begin{pmatrix}
    a_k \\
    a_{k-1} \\
    a_{k-2}
\end{pmatrix} &= \frac{1}{6}\begin{pmatrix}
    4 & 1 & 1 \\
    2 & 1 & -1 \\
    1 & 1 & 1
\end{pmatrix}\begin{pmatrix}
    2 & 0 & 0 \\
    0 & 1 & 0 \\
    0 & 0 & -1 
\end{pmatrix}^{k-3}\begin{pmatrix}
    2 & 0 & -2 \\
    -3 & 3 & 6 \\
    1 & -3 & 2 
\end{pmatrix}\begin{pmatrix}
    14\\
    6 \\
    2
\end{pmatrix} \\
&= \begin{pmatrix}
    4 & 1 & 1 \\
    2 & 1 & -1 \\
    1 & 1 & 1
\end{pmatrix}\begin{pmatrix}
    2^{k-3} & 0 & 0 \\
    0 & 1 & 0 \\
    0 & 0 & (-1)^{k-3} 
\end{pmatrix}\begin{pmatrix}
    4\\
    -2 \\
    0
\end{pmatrix} \\
&= \begin{pmatrix}
    4 & 1 & 1 \\
    2 & 1 & -1 \\
    1 & 1 & 1
\end{pmatrix}\begin{pmatrix}
    2^{k-1}\\
    -2 \\
    0
\end{pmatrix}.
\end{align*}
Extracting the leading entry of the resulting vector gives $a_k = 2^{k+1} - 2$. 
\end{proof}

It remains to identify the recursive description of $N_{k,1}$ in \cref{lemma:a_n}.

\begin{proof}[Proof of \cref{lemma:a_n}.]
The initial condition $a_1 = 2$ is immediate. The counts $a_2 = 6$ and $\frac{1}{2}a_3 = 7$ are provided in \cref{fig:a_n_initial}; reflective symmetry across the horizontal axis provides the other half of $a_3 = 14$. 

\begin{figure}[ht]
	\centering
	\begin{overpic}[scale=.4]{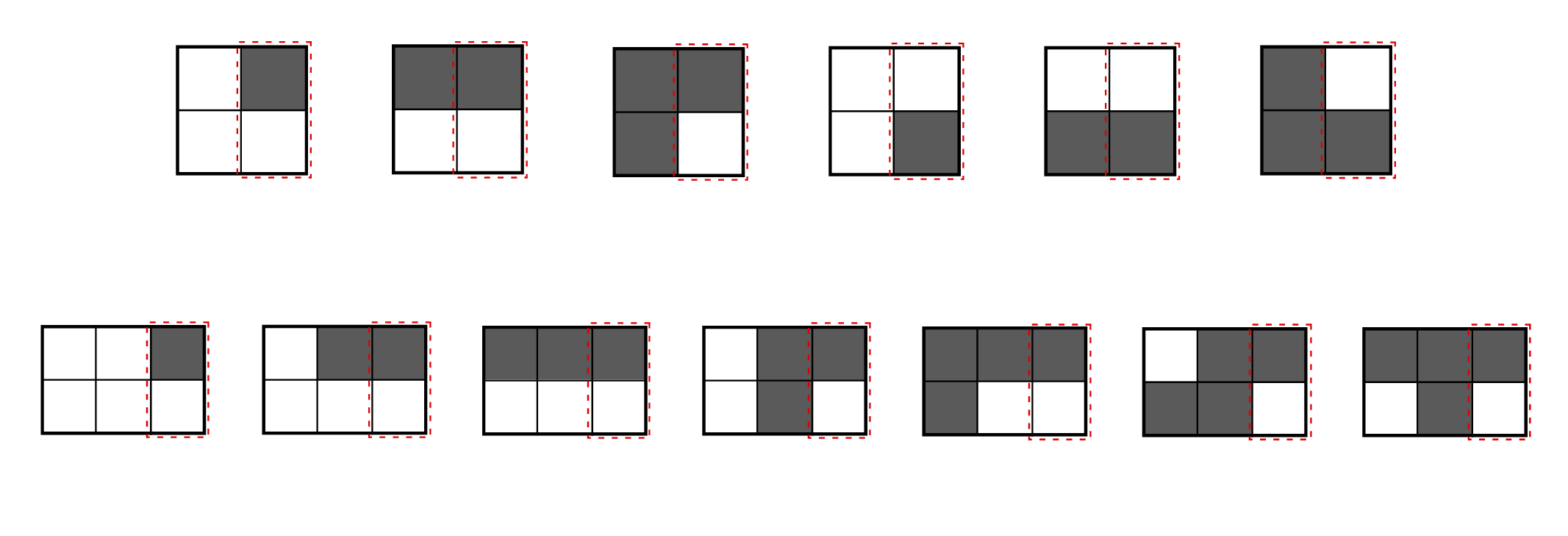}
        
	\end{overpic}
    \vskip-0.5cm
	\caption{The top row depicts $a_2 = 6$, and the bottom row depicts $\frac{1}{2}a_3 = 7$.}
	\label{fig:a_n_initial}
\end{figure}

Next, we claim that $a_k$ satisfies the nonhomogeneous second-order recursion  
\begin{equation}\label{eq:a_n_second_order}
a_k = a_{k-1} + 2a_{k-2} + 4.
\end{equation}
The desired third-order homogeneous recursion follows from the difference of successive terms:
\begin{align*}
    a_k - a_{k-1} &= (a_{k-1} + 2a_{k-2} + 4) - (a_{k-2} + 2a_{k-3} + 4) \\
    &= a_{k-1} + a_{k-2} - 2a_{k-3}.
\end{align*}
hence $a_k = 2a_{k-1} + a_{k-2} - 2a_{k-3}$. We dedicate the rest of the proof to establishing \eqref{eq:a_n_second_order}. 

Consider a Nurikabe rectangle in $\mathcal{N}_{k,1}$, assuming without loss of generality that $k\geq 3$. By definition, column $k$ has precisely one water square. Note that deleting column $k$ gives an element --- called the \textit{column reduction} --- of $\mathcal{N}_{k-1}$, as such a column deletion creates no whirlpools and cannot destroy connectedness of the water. Conversely, an invalid $2\times (k-1)$ rectangle cannot extend to an element of $\mathcal{N}_{k,1}$, as appending a rightmost column cannot connect disconnected water or destroy a whirlpool. 

By connectedness of the water on the $2\times k$ rectangle, the water square in column $k$ is either the only water square, or column $k-1$ has at least one water square. The former case accounts for $2$ elements of $\mathcal{N}_{k,1}$, as there are two possibilities for which of the two rightmost squares is water. In the latter case, if column $k-1$ has precisely one water square, it must be adjacent to the water square in column $k$ to preserve connectivity; on the other hand, each rectangle in $\mathcal{N}_{k-1,2}$ has two preimages in $\mathcal{N}_{k,1}$ under column reduction. Thus 
\begin{equation}\label{eq:loop1}
N_{k,1} = 2 + N_{k-1,1} + 2N_{k-1,2}.    
\end{equation}
Next we analyze the $N_{k-1,2}$ term. As both squares in column $k-1$ of a rectangle in $\mathcal{N}_{k-1,2}$ are water, connectedness of the water means they are either the only water squares, or column $k-2$ has precisely one water square, lest there be a whirlpool. As an invalid $2\times (k-2)$ rectangle cannot extend to a valid element of $\mathcal{N}_{k-1,2}$, we have 
\begin{equation}\label{eq:loop2}
   N_{k-1,2} = 1 + N_{k-2,1}.
\end{equation}
Note that \eqref{eq:loop2} gives the first equality of the second claim in the statement of the proposition. Combining \eqref{eq:loop1} and \eqref{eq:loop2}, we therefore have
\[
N_{k,1} = 2 + N_{k-1,1} + 2(1 + N_{k-2,1}) = 4 + N_{k-1,1} + 2N_{k-2,1}
\]
which is the desired recursion \eqref{eq:a_n_second_order}. This completes the proof.
\end{proof}

\subsection{Bicolumn restriction}

For $k\geq 2$ and $i,j\in\{0,1,2\}$, let $_j\mathcal{N}_{k,i}$ denote the set of $2\times k$ Nurikabe rectangles such that $i$ squares in column $k$ are water and $j$ squares in column $1$ are water. Let $_jN_{k,i}:= |_j\mathcal{N}_{k,i}|$.

\begin{lemma}\label{lemma:jacobsthal}
For $k\geq 2$ we have $_2N_{k+1,2} = \,_2N_{k,1} = 2J_k$ where $J_k = \frac{2^k - (-1)^k}{3}$ is the Jacobsthal sequence.     
\end{lemma}

\begin{proof}
Let $R\in \,_2\mathcal{N}_{k+1,2}$. As the water is connected, column $k$ must contain at least one water square. The lack of whirlpools implies that in fact there must be precisely one water square. Thus, column reduction gives a bijection $_2\mathcal{N}_{k+1,2} \to \,_2\mathcal{N}_{k,1}$.

Now let $b_k := \,_2N_{k,1}$ for $k\geq 2$. By counting directly, we see $b_2 = b_3 = 2$. We claim that $b_k$ satisfies the linear recurrence 
\begin{equation}\label{eq:jacobsthal}
b_k = b_{k-1} + 2b_{k-2}.    
\end{equation}
This recurrence with the above initial conditions implies $b_k = 2J_k$, where $J_k$ is the Jacobsthal sequence with known closed form as indicated; see \cite[A001045]{oeis}. To see \eqref{eq:jacobsthal}, note that for rectangles in $_2\mathcal{N}_{k,1}$ and $k\geq 4$, connectedness of the water implies there are two possibilities for column $k-1$: either this column has precisely one water square adjacent to the water square in column $k$, or both squares in the column are water. In the former case, column reduction gives a bijection of such rectangles with $_2\mathcal{N}_{k-1,1}$; in the latter case, column reduction gives a $2$-to-$1$ map of such rectangles onto $_2\mathcal{N}_{k-1,2}$; see \cref{fig:left2}. Consequently we have $_2N_{k,1} =\, _2N_{k-1,1} + 2 (_2N_{k-1,2})$. By the first paragraph of the proof, $_2N_{k-1,2} =\, _2N_{k-2,1}$ and thus 
\[
_2N_{k,1} =\, _2N_{k-1,1} + 2 (_2N_{k-2,1})
\]
which is precisely \eqref{eq:jacobsthal}.
\end{proof}

\begin{figure}[ht]
	\centering
	\begin{overpic}[scale=.44]{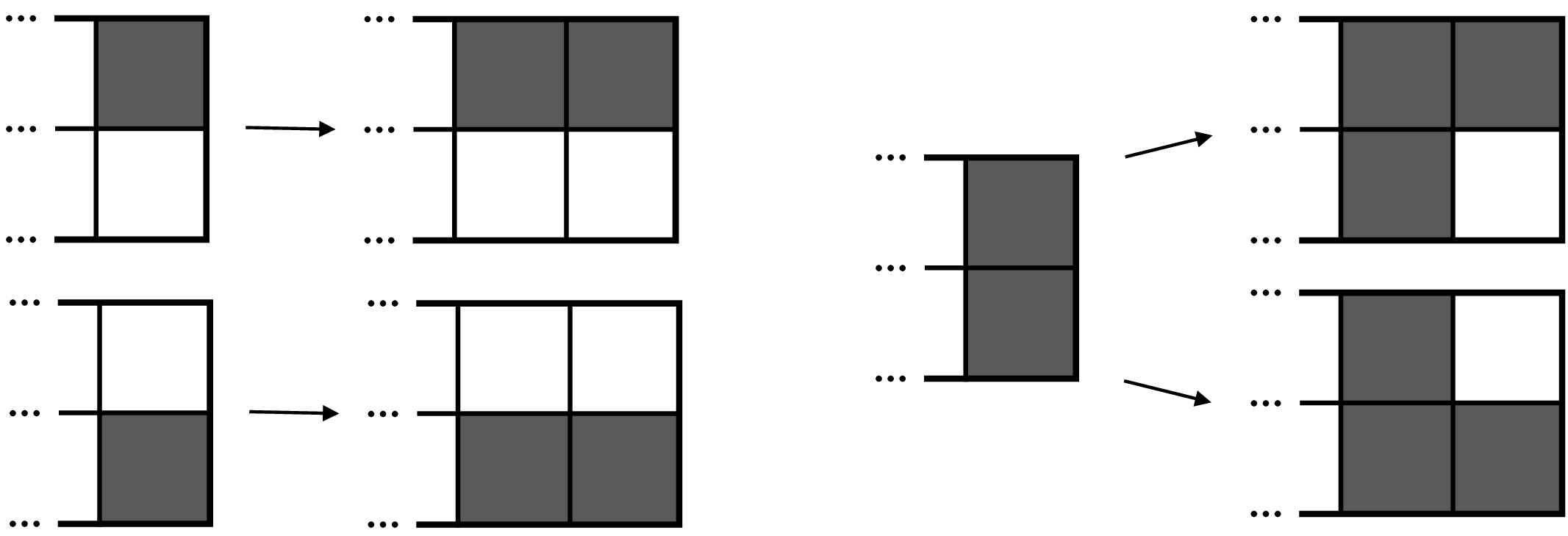}
      
	\end{overpic}
    
	\caption{Identifying the recurrence in \cref{lemma:jacobsthal}.} 
	\label{fig:left2}
\end{figure}
\section{Square whirlpool Nurikabe}\label{sec:square_whirl}

Here we establish \eqref{eq:main_square} of \cref{thm:main}. We first show that with the square whirlpool rule, our desired count is independent of the choice of non-orientable surface and thus it suffices to consider Möbius strips. 

\begin{lemma}
    For $n\geq 1$, we have $\accentset{\Box}{M}_n = \accentset{\Box}{K}_n = \accentset{\Box}{P}_n$.
\end{lemma}

\begin{proof}
By \cref{lemma:KPM_connected}, connectedness of water is equivalent on all three surfaces. We claim that (non)existence of square whirlpools is also independent of whether the fundamental domain is glued to a Möbius strip, Klein bottle, or projective plane. Specifically, the sets of interior vertices of square-degree $4$ coincide for all three surfaces: in the Möbius strip, the four corners of the fundamental domain correspond to two boundary vertices, and in the two closed surfaces they become a single interior vertex with square-degree $2$; all other vertices are unaffected by the choice of vertical edge (non)identification. It follows that square Nurikabe validity of a coloring is independent of the three surfaces. Hence, $\accentset{\Box}{M}_n = \accentset{\Box}{K}_n = \accentset{\Box}{P}_n$.
\end{proof}

\begin{lemma}\label{lemma:even_M_square_bij}
Rectangular reduction induces a bijection $\mathrm{red}: \accentset{\Box}{\mathcal{M}}_{2k} \to \mathcal{N}_k.$   
\end{lemma}

\begin{proof}
We first verify that $\mathrm{red}$ indeed has target $\mathcal{N}_k$. By \cref{lemma:connected_red}, rectangular reduction preserves connectedness of the water; moreover, it creates no new interior vertices and thus creates no new whirlpools. It follows that $\mathrm{red}(\accentset{\Box}{\mathcal{M}}_{2k}) \subset \mathcal{N}_k$. The map is injective by construction, so it remains to establish surjectivity. For $R\in \mathcal{N}_k$, let $M$ be the colored Möbius strip with $\mathrm{red}(M) = R$; we need to check that $M\in \accentset{\Box}{\mathcal{M}}_{2k}$. Again by \cref{lemma:connected_red} we know the water is connected on $M$. By \cref{lemma:orientable_rect} there are no whirlpools on the annulus obtained by identifying the horizontal edges of $R$; the final edge identification to produce $M$ only creates one additional interior vertex with square-degree $2$, hence it cannot become a square whirlpool. Thus, $M\in \accentset{\Box}{\mathcal{M}}_{2k}$ as desired. 
\end{proof}

\begin{lemma}\label{lemma:odd_square_M_bij}
There is a bijection
\[
\mathrm{red} \circ \mathrm{contr}: \accentset{\Box}{\mathcal{M}}_{2k+1} \to \mathcal{N}_k \sqcup \left(\mathcal{N}_{k,1} \cup \mathcal{N}_{k,2} \cup \{R_{\mathrm{land}}\}\right)
\]
where $R_{\mathrm{land}}\in \mathcal{N}_k$ is the $2\times k$ Nurikabe rectangle consisting of all land squares. 
\end{lemma}

\begin{proof}
As in \cref{lemma:odd_M_both}, decompose $\accentset{\Box}{\mathcal{M}}_{2k+1} = \accentset{\Box}{\mathcal{M}}_{2k+1,0} \cup \accentset{\Box}{\mathcal{M}}_{2k+1,1}$, where the first (resp. second) subset consists of Nurikabe Möbius strips whose central square is land (resp. water). The referenced lemma gives a bijection 
\[
\mathrm{red} \circ \mathrm{contr}: \accentset{\Box}{\mathcal{M}}_{2k+1,0} \to \mathcal{N}_k.
\]
For the remaining subset, we establish a bijection 
\[
\mathrm{red} \circ \mathrm{contr}: \accentset{\Box}{\mathcal{M}}_{2k+1,1} \to \mathcal{N}_{k,1} \cup \mathcal{N}_{k,2} \cup \{R_{\text{land}}\}.
\]
Let $M_0\in \accentset{\Box}{\mathcal{M}}_{2k+1,1}$ denote the Nurikabe Möbius strip with precisely one water square (necessarily the central $k+1$ square); we have $(\mathrm{red} \circ \mathrm{contr})(M_0) = R_{\text{land}}$. Thus, we may assume without loss of generality that $k\geq 2$ and it remains to verify the bijection 
\begin{equation}\label{eq:MS1}
\mathrm{red} \circ \mathrm{contr}: \accentset{\Box}{\mathcal{M}}_{2k+1,1} \setminus M_0 \to \mathcal{N}_{k,1} \cup \mathcal{N}_{k,2}.    
\end{equation}
We begin by checking that $\mathrm{red} \circ \mathrm{contr}$ has target $\mathcal{N}_{k,1} \cup \mathcal{N}_{k,2}$ when restricted to $\accentset{\Box}{\mathcal{M}}_{2k+1,1} \setminus M_0$. For $M\in \accentset{\Box}{\mathcal{M}}_{2k+1,1} \setminus M_0$, at least one of the squares in position $k,k+2$ must be water: if they were both land, the water would be disconnected, as by assumption there are other water squares beyond square $k+1$.\footnote{Both squares $k$ and $k+2$ may plausibly be water, as the interior vertices adjacent to square $k+1$ have square-degree $3$ and thus cannot be square whirlpools.} Connectedness of the water and lack of square whirlpools persist under $\mathrm{contr}$, as contracting the central square creates no interior vertices of square degree $4$ and only deletes the $k+1$ vertex (and its incident edges) from the adjacency graph. It follows that $\mathrm{contr}(M)\in \accentset{\Box}{\mathcal{M}}_{2k}$ and moreover at least one of its central $k,k+1$ squares is water. By \cref{lemma:even_M_square_bij}, $\mathrm{red}(\mathrm{contr}(M)) \in \mathcal{N}_{k,1} \cup \mathcal{N}_{k,2}$ as desired. 

The map in \eqref{eq:MS1} is injective by construction, so we only need to check surjectivity. For $R\in \mathcal{N}_{k,1} \cup \mathcal{N}_{k,2}$, let $M':= \mathrm{red}^{-1}(R) \in \accentset{\Box}{\mathcal{M}}_{2k}$ be the square Nurikabe Möbius strip under the bijection in \cref{lemma:even_M_square_bij}. Let $M$ be the $1\times (2k+1)$ colored Möbius strip obtained from $M'$ by reversing contraction with a central water square. Since at least one of the $k,k+1$ squares in $M'$ is water, this additional water square in $M$ does not disconnect the water. No square whirlpools are created, as reversing contraction creates no additional interior vertices of square degree $4$. Thus, $M\in \accentset{\Box}{\mathcal{M}}_{2k+1,1} \setminus M_0$ and $(\mathrm{red} \circ \mathrm{contr})(M) = R$. This establishes the bijection \eqref{eq:MS1} and completes the proof. 
\end{proof}

\begin{proof}[Proof of \eqref{eq:main_square} of \cref{thm:main}.]
The bijections in \cref{lemma:even_M_square_bij} and \cref{lemma:odd_square_M_bij} give the counts $\accentset{\Box}{M}_{2k} = N_{k}$ and $\accentset{\Box}{M}_{2k+1} = N_k + N_{k,1} + N_{k,2} + 1$. In the latter case, \cref{prop:a_n} further gives 
\begin{align*}
{M}_{2k+1} &= N_k + N_{k,1} + N_{k,2} + 1 \\
&= N_k + (2^{k+1} - 2) + (2^k - 1) + 1 \\
&= N_k + 2^{k+1} + 2^k - 2 \\
&= N_k + 3\cdot 2^k - 2.
\end{align*}\end{proof}

\section{Loop whirlpool Nurikabe}\label{sec:loop_whirl}

In this section we complete the proof of \cref{thm:main}. \cref{subsec:mobius_loop} establishes the Möbius strip count in \eqref{eq:main_loop}, and \cref{subsec:klein_loop} considers the case of Klein bottles and projective planes to obtain \eqref{eq:main_loop2}.

\subsection{Möbius strips}\label{subsec:mobius_loop}

We first consider Möbius strips of even length. 

\begin{lemma}\label{lemma:mobius_loop_red_even}
For $k\geq 1$, there is a bijection
\[
\mathrm{red}: \accentset{\ocircle}{\mathcal{M}}_{2k} \to \mathcal{N}_{k,0} \cup \mathcal{N}_{k,1}
\]
and thus $\accentset{\ocircle}{M}_{2k} = N_{k,0} + N_{k,1} = N_{k} - N_{k,2}$.     
\end{lemma}

\begin{proof}
First, we verify that if $M\in \accentset{\ocircle}{\mathcal{M}}_{2k}$, its rectangular reduction $\mathrm{red}(M)$ is an element of $\mathcal{N}_{k,0} \cup \mathcal{N}_{k,1}$. By \cref{lemma:connected_red}, connectedness of the water on $M$ persists in $\mathrm{red}(M)$, and as cutting operations produce no new interior vertices the lack of whirlpools on $M$ also persists in $\mathrm{red}(M)$. These observations imply that $\mathrm{red}(M) \in \mathcal{N}_k$. That $\mathrm{red}(M) \in \mathcal{N}_{k,0} \cup \mathcal{N}_{k,1}$ follows from the observation that if $\mathrm{red}(M)\in \mathcal{N}_{k,2}$ then $M$ has a (loop) whirlpool. Indeed, the central (interior) vertex is entirely surrounded by water in $M$, as shown in \cref{fig:redpool}. 

\begin{figure}[ht]
	\centering
	\begin{overpic}[scale=.44]{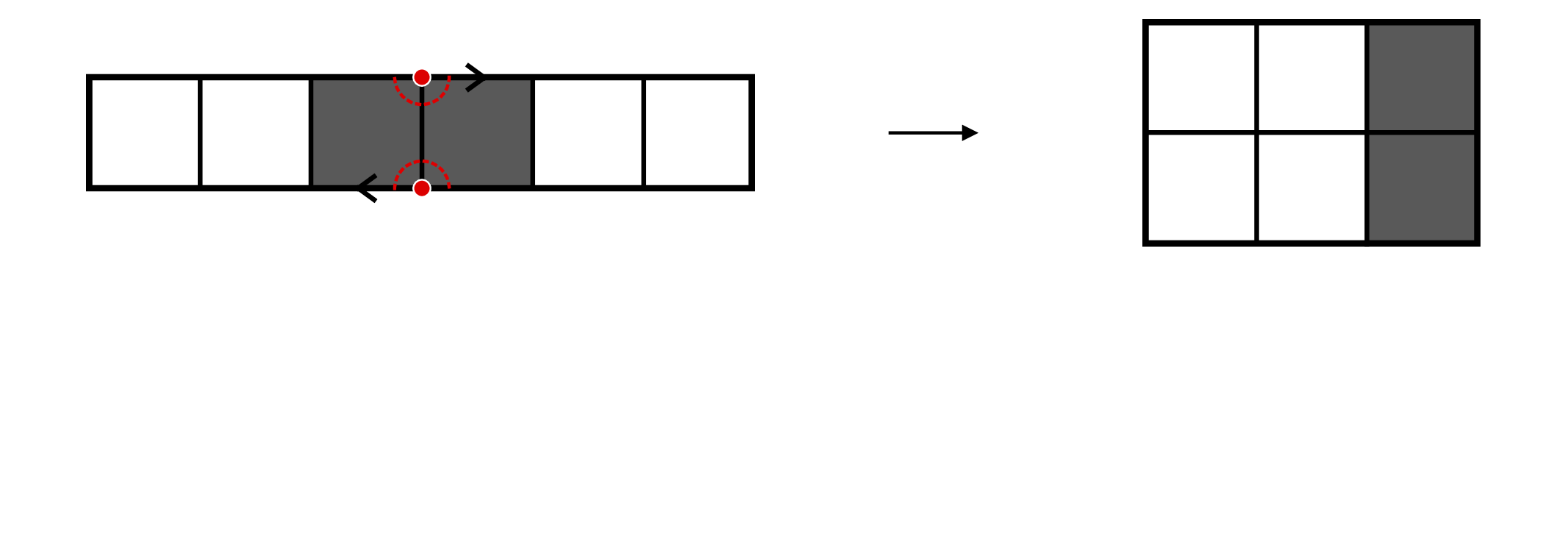}
      
	\end{overpic}
    \vskip-2.5cm
	\caption{A Möbius strip whose central interior vertex (in red) is surrounded by water yields a rectangle under rectangular reduction whose rightmost column is entirely water.} 
	\label{fig:redpool}
\end{figure}

The map $\mathrm{red}: \accentset{\ocircle}{\mathcal{M}}_{2k} \to \mathcal{N}_{k,0} \cup \mathcal{N}_{k,1}$ is injective by construction. To see surjectivity it suffices to verify that any $R\in \mathcal{N}_{k,0} \cup \mathcal{N}_{k,1}$ glues to a valid Nurikabe Möbius strip under the inverse of rectangular reduction. Again, by \cref{lemma:connected_red} we only need to check that no (loop) whirlpools are created in the gluing process. By \cref{lemma:orientable_rect}, all interior vertices created when producing an annulus are surrounded by $2\times 2$ sub-grids that already surrounded interior vertices on the rectangle, thus no whirlpool is created at this step. When gluing the final pair of edges to produce the Möbius strip, only one new interior vertex is created. However, since $R\notin \mathcal{N}_{k,2}$, this interior vertex is incident to at least one land square. Thus, $\mathrm{red}^{-1}(R)$ has no whirlpools. This establishes the bijection and verifies the first equality in the lemma; the second is true by rearranging the (definitional) fact that $N_k = N_{k,0} + N_{k,1} + N_{k,2}$.
\end{proof}

Next we consider Möbius strips of odd length. 

\begin{lemma}\label{lemma:mobius_loop_odd}
For $k\geq 1$, there is a bijection 
\[
\mathrm{red} \circ \mathrm{contr}: \accentset{\ocircle}{\mathcal{M}}_{2k+1} \to \mathcal{N}_k \sqcup \left(\mathcal{N}_{k,1} \cup \{R_{\mathrm{land}}\}\right)
\]
and thus $\accentset{\ocircle}{M}_{2k+1} = N_k + N_{k,1} + 1 = N_k + N_{k+1,2}$.    
\end{lemma}

\begin{proof}
Decompose $\accentset{\ocircle}{\mathcal{M}}_{2k+1} = \accentset{\ocircle}{\mathcal{M}}_{2k+1,0} \cup \accentset{\ocircle}{\mathcal{M}}_{2k+1,1}$, where the first (resp. second) subset consists of Nurikabe Möbius strips whose central square is land (resp. water). \cref{lemma:odd_square_M_bij} gives a bijection 
\[
\mathrm{red} \circ \mathrm{contr}: \accentset{\ocircle}{\mathcal{M}}_{2k+1,0} \to \mathcal{N}_k.
\]
Next we claim that there is a bijection 
\[
\mathrm{red} \circ \mathrm{contr}: \accentset{\ocircle}{\mathcal{M}}_{2k+1,1} \to \{R_{\text{land}}\} \cup \mathcal{N}_{k,1}.
\]
The proof is similar to that of \cref{lemma:odd_square_M_bij} with slight modifications due to the loop whirlpool rule. Let $M_0\in \accentset{\ocircle}{\mathcal{M}}_{2k+1,1}$ denote the Nurikabe Möbius strip with precisely one (central) water square, so that $(\mathrm{red} \circ \mathrm{contr})(M_0) = R_{\text{land}}$. We only need to verify the bijection 
\[
\mathrm{red} \circ \mathrm{contr}: \accentset{\ocircle}{\mathcal{M}}_{2k+1,1} \setminus M_0 \to \mathcal{N}_{k,1}.
\]
For $M\in \accentset{\ocircle}{\mathcal{M}}_{2k+1,1} \setminus M_0$, exactly one of the squares adjacent to the central water square must be water: if they were both water, there would be a loop whirlpool, and if they were both land, then the water would be disconnected, as by assumption there are other water squares beyond the central square. Connectedness of the water and lack of whirlpools then persist under $\mathrm{contr}$, as contracting the central square only affects adjacencies and interior vertices of the squares immediately adjacent to it. By \cref{lemma:mobius_loop_red_even} it follows that $\mathrm{red}(\mathrm{contr}(M))\in \mathcal{N}_{k,1}$. 

As usual, injectivity follows by construction. To see surjectivity, for $R\in \mathcal{N}_{k,1}$ we let $M' = \mathrm{red}^{-1}(R)$ denote the corresponding even Möbius strip, which is Nurikabe valid by \cref{lemma:mobius_loop_red_even}. Let $M$ be the odd Möbius strip obtained from $M'$ by reversing contraction with a central water square. Adding an additional water square cannot disconnect the water on $M$, and no whirlpools are created because one of the squares adjacent to the newly-created central water square is land. Thus, $M\in \accentset{\ocircle}{\mathcal{M}}_{2k+1,1} \setminus M_0$ and $(\mathrm{red}\circ \mathrm{contr})(M) = R$. 

The two bijections 
\begin{align*}
\accentset{\ocircle}{\mathcal{M}}_{2k+1,0} &\to \mathcal{N}_k \\
    \accentset{\ocircle}{\mathcal{M}}_{2k+1,1} &\to \{R_{\text{land}}\} \cup \mathcal{N}_{k,1}
\end{align*}
then imply 
\[
\accentset{\ocircle}{M}_{2k+1} = |\accentset{\ocircle}{\mathcal{M}}_{2k+1,0}| + |\accentset{\ocircle}{\mathcal{M}}_{2k+1,1}| = N_k + N_{k,1} + 1
\]
as desired. The second equality in the statement of the lemma is then a consequence of \cref{prop:a_n}.
\end{proof}

\begin{proof}[Proof of \eqref{eq:main_loop} of \cref{thm:main}.]
By \cref{lemma:mobius_loop_red_even} and \cref{prop:a_n} we have 
\[
\accentset{\ocircle}{M}_{2k} = N_k - N_{k,2} = N_k - (2^k - 1) = N_k - 2^k + 1.
\]
Similarly, \cref{lemma:mobius_loop_odd} and \cref{prop:a_n} combine to give 
\[
\accentset{\ocircle}{M}_{2k+1} = N_k + N_{k+1,2} = N_k + (2^{k+1} - 1) = N_k + 2^{k+1} - 1.
\]
This completes the proof. 
\end{proof}

\subsection{Klein bottles and projective planes}\label{subsec:klein_loop}

Finally, we consider $1\times n$ Klein bottles and projective planes as depicted in the second and third rows of \cref{fig:nurikabe3}, establishing \eqref{eq:main_loop2} of \cref{thm:main}. We first observe that $\accentset{\ocircle}{K}_n = \accentset{\ocircle}{P}_n$, as the choice of gluing direction for the vertical edges does not affect the adjacency graph, nor does it affect the local structure of the interior vertex corresponding to the four corners of the fundamental domain, which is a loop whirlpool (in either case) if and only if squares $1$ and $n$ are both water. Thus, we restrict without loss of generality to the Klein bottle.

Our next observation is the following. Let $_2\accentset{\ocircle}{\mathcal{M}}_n$ denote the collection of loop Nurikabe Möbius strips of size $1\times n$ such that the squares $1$ and $n$ are water, and let $_2\accentset{\ocircle}{M}_n:=|_2\accentset{\ocircle}{\mathcal{M}}_n|$.

\begin{lemma}\label{lemma:klein1}
For $n\geq 2$, we have $\accentset{\ocircle}{K}_n =\accentset{\ocircle}{M}_n - \,_2\accentset{\ocircle}{M}_n$.  
\end{lemma}

\begin{proof}
By \cref{lemma:KPM_connected}, connectedness of water is equivalent when viewed on the Klein bottle or Möbius strip. The only change in gluing to a Klein bottle is that the vertex corresponding to the four corners of the fundamental domain becomes an interior vertex, rather than a boundary vertex as on the Möbius strip. In particular, the only qualitative difference on the Klein bottle is the possibility that this interior vertex is a loop whirlpool, which happens precisely when squares $1$ and $n$ are water. Thus, we count $\accentset{\ocircle}{K}_n$ by subtracting from $\accentset{\ocircle}{M}_n$ the number of such Möbius strips. 
\end{proof}

\begin{proof}[Proof of \eqref{eq:main_loop2} of \cref{thm:main}.]
The counts $\accentset{\ocircle}{K}_{1}, \accentset{\ocircle}{K}_{2}$, $\accentset{\ocircle}{K}_{3}$, and $\accentset{\ocircle}{K}_{4}$ are explicit enumerations which are necessary because of the range of the $k$-index in \cref{lemma:jacobsthal}. The case $\accentset{\ocircle}{K}_{4} = 7$ is provided in \cref{fig:K4}, along with the associated counts $\accentset{\ocircle}{M}_{4} = 10$ and $_2\accentset{\ocircle}{M}_{4} = 3$ to clarify the relationship described in \cref{lemma:klein1}.

\begin{figure}[ht]
	\centering
	\begin{overpic}[scale=.44]{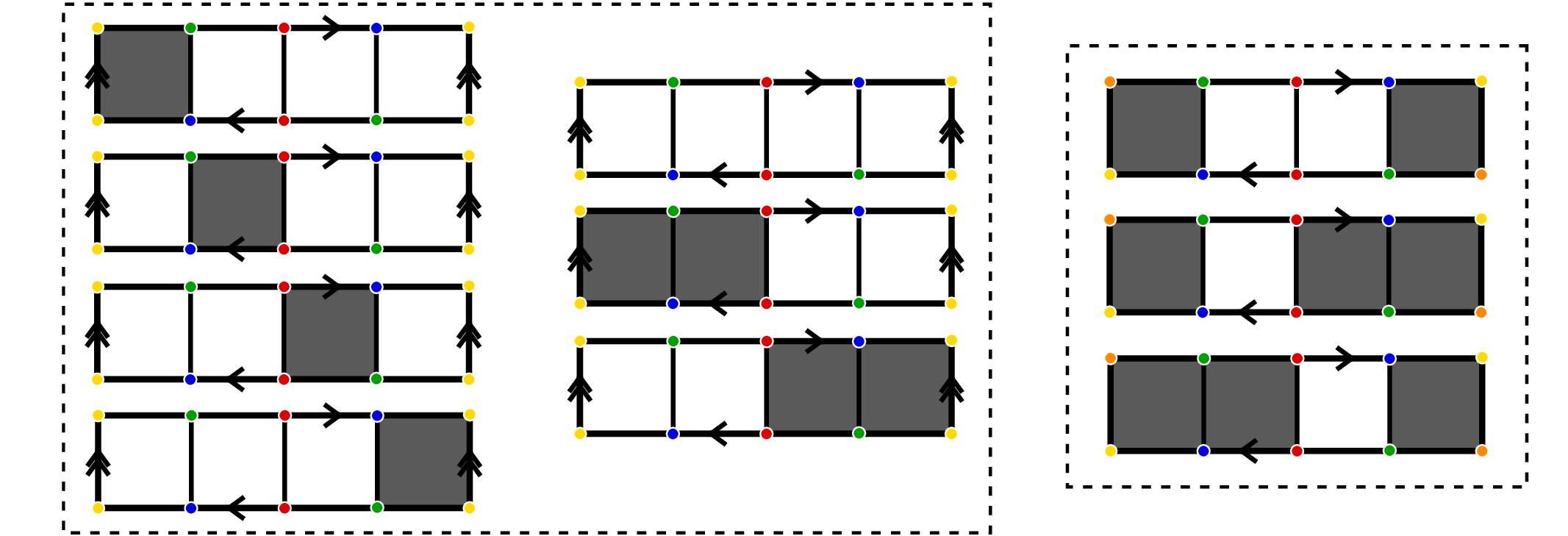}
      
	\end{overpic}
    
	\caption{The case $n=4$ for the loop Klein bottle count (left).} 
	\label{fig:K4}
\end{figure}

Now assume $n\geq 5$. Recall that \cref{lemma:mobius_loop_red_even} gives a bijection $\mathrm{red}:\accentset{\ocircle}{\mathcal{M}}_{2k} \to \mathcal{N}_{k,0} \cup \mathcal{N}_{k,1}$ and thus the count $\accentset{\ocircle}{M}_{2k} = N_k - N_{k,2}$. Filtering for rectangles whose leftmost column is all water, it follows that 
\begin{align*}
_2\accentset{\ocircle}{M}_{2k} &=\, _2N_k - \, _2N_{k,2} \\
&= N_{k,2} - 2J_{k-1}.
\end{align*}
Here, $_2N_k$ counts the number of $2\times k$ Nurikabe rectangles whose first column is entirely water; we have $_2N_k = N_{k,2}$ by reflection, hence the second equality above follows from this and \cref{lemma:jacobsthal}. Then, by \cref{lemma:klein1} and \cref{prop:a_n} we have 
\begin{align*}
\accentset{\ocircle}{K}_{2k} &=\accentset{\ocircle}{M}_{2k} - \,_2\accentset{\ocircle}{M}_{2k}    \\
&= (N_k - N_{k,2}) - (N_{k,2} - 2J_{k-1}) \\
&= N_k - 2N_{k,2} + 2J_{k-1} \\
&= N_k  + 2J_{k-1} - 2(2^k - 1) \\
&= N_k  + 2J_{k-1} - 2^{k+1} +2.
\end{align*}
Similarly, the bijection 
\[
\accentset{\ocircle}{\mathcal{M}}_{2k+1} = \accentset{\ocircle}{\mathcal{M}}_{2k+1,0} \cup \accentset{\ocircle}{\mathcal{M}}_{2k+1,1} \to \mathcal{N}_k \sqcup \left(\mathcal{N}_{k,1} \cup \{R_{\text{land}}\}\right)
\]
from \cref{lemma:mobius_loop_odd} gave the count $\accentset{\ocircle}{M}_{2k+1} = N_k + N_{k,1} + 1$. After filtering for those rectangles with a first column consisting of two water squares and applying \cref{lemma:jacobsthal}, we have 
\begin{align*}
_2\accentset{\ocircle}{M}_{2k+1} &=\, _2N_k +\, _2N_{k,1} + 0 \\
&= N_{k,2} + 2J_k.
\end{align*}
Therefore, 
\begin{align*}
\accentset{\ocircle}{K}_{2k+1} &=\accentset{\ocircle}{M}_{2k+1} - \,_2\accentset{\ocircle}{M}_{2k+1}    \\    
&= (N_k + 2^{k+1} - 1) - (N_{k,2} + 2J_k) \\
&= N_k - 2J_k + (2^{k+1} - 1) - (2^k - 1) \\
&= N_k - 2J_k + 2^k,
\end{align*}
completing the proof.
\end{proof}

\bibliography{references}
\bibliographystyle{amsalpha}

\end{document}